\documentclass[12pt,reqno]{amsart}

\usepackage{pdfsync}

\usepackage{amssymb}

\usepackage{mathrsfs}

\DeclareMathAlphabet{\mathpzc}{OT1}{pzc}{m}{it}

\usepackage[all]{xy}

\entrymodifiers={+!!<0pt,\fontdimen22\textfont2>}

\def\cP{\mathscr{P}}
\def\cQ{\mathscr{Q}}
\def\cR{\mathscr{R}}
\def\cS{\mathscr{S}}

\def\BC{\mathbb{C}}

\def\BR{\mathbb{R}}

\def\BZ{\mathbb{Z}}

\def\sA{\mathsf{A}}
\def\sB{\mathsf{B}}
\def\sC{\mathsf{C}}
\def\sD{\mathsf{D}}

\def\sT{\mathsf{T}}
\def\sU{\mathsf{U}}
\def\sV{\mathsf{V}}

\def\add{\operatorname{add}}

\def\adots{\mathinner{\mkern1mu\raise1.0pt\vbox{\kern7.0pt\hbox{.}}\mkern2mu\raise4.0pt\hbox{.}\mkern2mu\raise7.0pt\hbox{.}\mkern1mu}}

\def\ast{{\textstyle *}}
\def\Aut{\operatorname{Aut}}

\def\Costab{\operatorname{Costab}}

\def\Coslice{\operatorname{Coslice}}
\def\D{\sD}
\def\Dc{\D^{\operatorname{c}}}

\def\GL{\operatorname{GL}}

\def\Hom{\operatorname{Hom}}
\def\id{\operatorname{id}}

\def\ind{\operatorname{ind}}
\def\inf{\operatorname{inf}}

\def\K{\operatorname{K}}

\def\rank{\operatorname{rank}}

\def\split{\operatorname{split}}

\newtheorem{Lemma}{Lemma}[section]
\newtheorem{Theorem}[Lemma]{Theorem}
\newtheorem{Proposition}[Lemma]{Proposition}

\theoremstyle{definition}
\newtheorem{Definition}[Lemma]{Definition}
\newtheorem{Setup}[Lemma]{Setup}

\newtheorem{Remark}[Lemma]{Remark}

\begin{document}

\setlength{\parindent}{0pt}
\setlength{\parskip}{7pt}

\title[The co-stability manifold of a triangulated category]{The co-stability manifold of a triangulated category}

\author{Peter J\o rgensen}
\address{School of Mathematics and Statistics,
Newcastle University, Newcastle upon Tyne NE1 7RU, United Kingdom}
\email{peter.jorgensen@ncl.ac.uk}
\urladdr{http://www.staff.ncl.ac.uk/peter.jorgensen}

\author{David Pauksztello}
\address{Institut f\"{u}r Algebra, Zahlentheorie und Diskrete
  Mathematik, Fa\-kul\-t\"{a}t f\"{u}r Ma\-the\-ma\-tik und Physik, Leibniz
  Universit\"{a}t Hannover, Welfengarten 1, 30167 Hannover, Germany}
\email{pauk@math.uni-hannover.de}
\urladdr{http://www.iazd.uni-hannover.de/\~{ }pauksztello}


\keywords{Co-slicing, co-stability condition,
co-t-structure, split Harder-Narasimhan property}

\subjclass[2010]{18E30}

\begin{abstract} 

Stability conditions on triangulated categories were introduced by
Brid\-ge\-land as a `continuous' generalisation of t-structures.  The
set of locally-finite stability conditions on a triangulated category
is a manifold which has been studied intensively.

However, there are mainstream triangulated categories whose stability
manifold is the empty set.  One example is $\Dc \big( k[X] / ( X^2 )
\big)$, the compact derived category of the dual numbers over an
algebraically closed field $k$.

This is one of the motivations in this paper for introducing
co-stability conditions as a `continuous' generalisation of
co-t-structures.  Our main result is that the set of nice co-stability
conditions on a triangulated category is a manifold.  In particular,
we show that the co-stability manifold of $\Dc \big( k[X] / ( X^2 )
\big)$ is $\BC$.

\end{abstract}

\maketitle

\setcounter{section}{0}
\section{Introduction}
\label{sec:introduction}

Triangulated categories are useful in several branches of mathematics,
and stability conditions are an important tool for their study
introduced by Bridgeland in \cite{B}.  Stability conditions are
`continuous' generalisations of bounded t-structures and the main
result of \cite{B} is that on a triangulated category, the set of
stability conditions which satisfy the technical condition of
local-finiteness is a manifold.  This `stability manifold' is divided
into subsets corresponding to bounded t-structures in the category.

However, there are mainstream triangulated categories for which the
stability manifold is the empty set.  An example is $\Dc \big( k[X] /
( X^2 ) \big)$, the compact derived category of the dual numbers over
an algebraically closed field $k$.  This is our first motivation for
introducing the `mirror' notion of co-stability conditions and proving
the following main theorem.

{\bf Theorem A. }
{\em
Let $\sT$ be a triangulated category satisfying the conditions in
Setup \ref{set:blanket} below.  Then the set of co-stability conditions on
$\sT$ which satisfy the technical condition in Definition
\ref{def:40} is a topological manifold. 
}

Indeed, the `co-stability manifold' of the category $\Dc \big( k[X] /
( X^2 ) \big)$ which exists by Theorem A is non-trivial:

{\bf Theorem B. }
{\em
Let $k$ be an algebraically closed field and consider $\Dc \big( k[X]
/ ( X^2 ) \big)$.  Its stability manifold is the empty set and its
co-stability manifold is $\BC$.  }

The co-stability manifold of a triangulated category is divided into
subsets cor\-re\-spon\-ding to bounded co-t-structures in the
category; see Remark \ref{rmk:chambers}.  Recall that co-t-structures
are, in a sense, a mirror image of t-structures.  They were introduced
independently in \cite[def.\ 1.1.1]{Bo2} and \cite[def.\ 2.4]{P}, see
Definition \ref{def:co-t-structure}, and have recently been the focus
of considerable interest, see \cite{AI}, \cite{Bo1}, \cite{Bo2},
\cite{Bo3}, \cite{HJY}, \cite{MSSS}, \cite{P}, \cite{S}, \cite{W}.
Alternatively, the co-stability manifold can be viewed as being divided into
subsets corresponding to silting subcategories as defined in
\cite[def.\ 2.1]{AI}, because these are in bijection with bounded
co-t-structures by \cite[cor.\ 4.7]{MSSS}.  These observations are our
second motivation for introducing co-stability conditions.

{\em Relation to Bridgeland's paper \cite{B}. }
Recall that a stability condition is a pair $( Z , \cP )$ where $Z :
\K_0( \sT ) \rightarrow \BC$ is a homomorphism, $\cP$ a so-called
slicing consisting of certain subcategories $\cP( \varphi )$ for
$\varphi \in \BR$.  It is required that $Z( p ) = m( p )\exp(
i\pi\varphi )$ with $m( p ) > 0$ for $p \in \cP( \varphi ) \setminus
0$.

We define co-stability conditions analogously, replacing the slicing
$\cP$ with a co-slicing $\cQ$; this notion is defined in Section
\ref{sec:co-slicings}.  Some other parts of what we do are also
closely inspired by \cite{B} as we shall point out along the way.

However, the passage from stability conditions to co-stability
conditions is non-trivial.  It is governed by a `looking glass
principle' (a term coined in \cite{AH}): Some results on stability
conditions have mirror versions for co-stability conditions, but
others do not and translation is rarely mechanical.  In fact, this is
already true of the passage from t-structures to co-t-structures.
This means that our proofs are different from those in \cite{B}.

{\em Further remarks and setup. }
We have chosen only to define the co-stability manifold for
triangulated categories with finitely generated $\K_0$-group.  This
covers the examples we have in mind from representation theory,
ensures that the co-stability manifold is finite dimensional, and
makes the theory less technical.

The paper is organised as follows: Section \ref{sec:co-t-structures}
recapitulates the definition of co-t-structures.  Section
\ref{sec:co-slicings} defines co-slicings in triangulated categories.
Section \ref{sec:metric} turns the set of co-slicings into a metric
space.  Section \ref{sec:co-stability_functions} defines co-stability
functions and the split Harder-Narasimhan property.  Section
\ref{sec:co-stability_conditions} defines co-stability conditions and
proves a crucial separation result in Proposition \ref{pro:27}.
Section \ref{sec:triangles} has two technical lemmas.  Section
\ref{sec:manifold} proves an equally crucial deformation result in
Proposition \ref{pro:deformation}; Theorem A is a consequence which
appears as Theorem \ref{thm:manifold}.  Section
\ref{sec:group_actions} remarks that, like the stability manifold, the
co-stability manifold admits commuting group actions of $\Aut( \sT )$
and $\widetilde{ \GL }^+( 2 , \BR )$.  Section \ref{sec:example1}
proves Theorem B which is a special case of Theorem
\ref{thm:example1}.  Section \ref{sec:example2} gives an example explaining
why the technical condition in Definition \ref{def:40} is necessary
for Proposition \ref{pro:deformation} and hence for Theorem A.

\begin{Setup}
\label{set:blanket}
Throughout, $\sT$ is an essentially small triangulated category which
is Krull-Schmidt and has finitely generated $\K_0(\sT)$.
\end{Setup}

When we say that $\sT$ is Krull-Schmidt, we mean that it has split
idempotents, that each object of $\sT$ is the direct sum of finitely
many indecomposable objects, and that each indecomposable object has
local endomorphism ring.  The Krull-Schmidt theorem then implies that
the indecomposable direct summands of a given object are determined up
to isomorphism.

We always assume that subcategories are closed under isomorphisms;
that is, if $a$ is an object of a subcategory and $a \cong a'$ in the
ambient category, then $a'$ is also in the sub\-ca\-te\-go\-ry.  Each
of our categorical closure operations is understood as producing full
subcategories.  In particular, $( \;\; )^-$ denotes closure under
extensions, $( \;\; )^+$ denotes closure under extensions and direct
summands, and $\add$ denotes closure under finite direct sums and direct
summands.  The symbol $\perp$ sends full subcategories of $\sT$ to
full subcategories as follows.
\[
  \sA^{\perp} = \{\, t \in \sT \,|\, \sT( \sA , t ) = 0 \,\}, \;\;\;
  {}^{\perp}\sB = \{\, t \in \sT \,|\, \sT( t , \sB ) = 0 \,\}.
\]
The prefix $\ind$ denotes the class of indecomposable objects in an
additive category.

We use $\sT( - , - )$ as shorthand for $\Hom_{ \sT }( - , - )$ and
denote the suspension functor of $\sT$ by $\Sigma$.  Distinguished
triangles are sometimes written in the form $\xymatrix @!  { t' \ar[r]
  & t \ar[r] & t'' \ar@{~>}[r] & t' }$; the wiggly arrow is short for
a morphism $t'' \rightarrow \Sigma t'$.  The ordinary Grothendieck
group is denoted by $\K_0$ and the split Grothendieck group by
$\K_0^{\split}$.

\section{Co-t-structures}
\label{sec:co-t-structures}

This section recalls the definition of co-t-structures and two useful
properties.  The definition is due independently to \cite[def.\
1.1.1]{Bo2} and \cite[def.\ 2.4]{P}; we have tweaked it slightly for
reasons of symmetry.

\begin{Definition}
\label{def:co-t-structure}
A co-t-structure in $\sT$ is a pair $( \sA , \sB )$ of full
subcategories closed under direct sums and summands satisfying the
following conditions.
\begin{enumerate}
\setcounter{enumi}{0}

  \item  $\Sigma^{-1}\sA \subseteq \sA$ and $\Sigma\sB \subseteq \sB$.

\smallskip

  \item  $\sT( \sA , \sB ) = 0$.

\smallskip

  \item  For each object $t \in \sT$ there is a distinguished triangle
  $a \rightarrow t \rightarrow b$ with $a \in \sA$, $b \in \sB$. 

\end{enumerate}
The co-heart is $\sC = \sA \cap \Sigma^{-1}\sB$.

The co-t-structure is called bounded if
\[
  \bigcup_{j \in \BZ} \Sigma^j \sA = \bigcup_{j \in \BZ} \Sigma^j \sB = \sT.
\]
\end{Definition}

\begin{Remark}
Note that if we replace (i) by the conditions that $\Sigma \sA
\subseteq \sA$ and $\Sigma^{-1}\sB \subseteq \sB$, then we get the
definition of a t-structure.
\end{Remark}

The following two propositions were proved in \cite[prop.\ 1.5.6 and
thm.\ 5.3.1]{Bo2}.  We restate them for the convenience of the
reader.  Note that Proposition \ref{pro:filtration} is the
co-t-structure analogue of \cite[lem.\ 3.2]{B}.

\begin{Proposition}
\label{pro:filtration}
Let $( \sA , \sB )$ be a bounded co-t-structure in $\sT$ with co-heart
$\sC = \sA \cap \Sigma^{-1}\sB$.  For each object $t \neq 0$ of $\sT$,
there is a diagram
\[
  \xymatrix @-2.5pc @! {
    0 \cong t_0 \ar[rr] & & t_1 \ar[rr] \ar[dl] & & t_2 \ar[rr] \ar[dl] & & \cdots \ar[rr] & & t_{n-1} \ar[rr] & & t_n \cong t \ar[dl] \\
    & \Sigma^{j_1}c_1 \ar@{~>}[ul] & & \Sigma^{j_2}c_2 \ar@{~>}[ul] & & & & & & \Sigma^{j_n}c_n \ar@{~>}[ul] & \\
               }
\]
consisting of distinguished triangles, where $c_m \in \sC$ for each
$m$ and $j_1 < j_2 < \cdots < j_n$.
\end{Proposition}

\begin{Proposition}
\label{pro:K}
Let $( \sA , \sB )$ be a bounded co-t-structure in $\sT$ with co-heart
$\sC$.  There is an isomorphism
\[
  \K_0^{\split}( \sC ) \stackrel{\sim}{\rightarrow} \K_0( \sT )
\]
given by $[c] \mapsto [c]$.

The inverse is $[t] \mapsto \sum_m [\Sigma^{j_m}c_m]$ where the
objects $\Sigma^{j_m}c_m$ come from a diagram as in Proposition
\ref{pro:filtration}; this sum determines a well-defined element
of $\K_0^{ \split }( \sC )$.
\end{Proposition}

\section{Co-slicings}
\label{sec:co-slicings}

This section introduces co-slicings.  They are a mirror image of the
slicings of \cite[def.\ 3.3]{B}.

\begin{Definition}
\label{def:4}
A co-slicing $\cQ$ in $\sT$ is a collection of full subcategories
$\cQ( \varphi )$ closed under direct sums and summands, indexed
by $\varphi \in \BR$ and satisfying the following conditions.
\begin{enumerate}
\setcounter{enumi}{0}

  \item  $\cQ( \varphi + 1 ) = \Sigma \cQ( \varphi )$.

\smallskip

  \item  $\varphi_1 < \varphi_2
\;\Rightarrow\; \sT \big( \cQ( \varphi_1 ) , \cQ( \varphi_2 ) \big) = 0$.

\smallskip

  \item  For each object $t \neq 0$ of $\sT$, there is a diagram
\[
  \xymatrix @-2.5pc @! {
    0 \cong t_0 \ar[rr] & & t_1 \ar[rr] \ar[dl] & & t_2 \ar[rr] \ar[dl] & & \cdots \ar[rr] & & t_{n-1} \ar[rr] & & t_n \cong t \ar[dl] \\
    & q_1 \ar@{~>}[ul] & & q_2 \ar@{~>}[ul] & & & & & & q_n \ar@{~>}[ul] & \\
                       }
\]
consisting of distinguished triangles, where $q_i \in \cQ( \varphi_i )$
and $\varphi_1 < \cdots < \varphi_n$.

\end{enumerate}
\end{Definition}

Note that (i) and (ii) are continuous versions of (i) and (ii) in
Definition \ref{def:co-t-structure} while (iii) is a continuous
version of Proposition \ref{pro:filtration}.

\begin{Lemma}
\label{lem:9}
Let $\cQ$ be a co-slicing in $\sT$ and consider the diagram from
Definition \ref{def:4}(iii).  For each $j$, there is an obvious
morphism $t_j \rightarrow t$ which we complete to a distinguished
triangle
\[
  t_j \rightarrow t \rightarrow e_j.
\]
Then for each $j$ there is a diagram
\[
  \xymatrix @-2.6pc @! {
    e_{j} \ar[rr] & & e_{ j+1 } \ar[rr] \ar@{~>}[dl] & & e_{ j+2 } \ar[rr] \ar@{~>}[dl] & & \cdots \ar[rr] & & e_{ n-1 } \ar[rr] & & e_n \cong 0 \ar@{~>}[dl] \\
    & q_{ j+1 } \ar[ul] & & q_{ j+2 } \ar[ul] & & & & & & q_n \ar[ul] & \\
                       }
\]
consisting of distinguished triangles.

This diagram and the one from Definition \ref{def:4}(iii) show
\[
  t_j \in 
  \big( \cQ( \varphi_1 ) \cup \cdots \cup \cQ( \varphi_j ) \big)^{-}, 
  \;\;\;
  e_j \in 
  \big( \cQ( \varphi_{ j+1 } ) \cup \cdots \cup \cQ( \varphi_n ) \big)^{-}.
\]
Recall that $( \;\; )^-$ denotes closure under extensions.
\end{Lemma}

\begin{proof}
We use descending induction on $j$.  The case $j = n - 1$ is clear.
The induction step is carried out by applying the octahedral axiom to
the composable morphisms $t_{ j-1 } \rightarrow t_j \rightarrow t$ to
get the following $3 \times 3$ diagram of distinguished triangles.
\[
  \xymatrix @-0.7pc @! {
    t_{ j-1 } \ar[r] \ar[d] & t_j \ar[r] \ar[d] & q_j \ar[d] \\
    t \ar@{=}[r] \ar[d] & t \ar[r] \ar[d] & 0 \ar[d] \\
    e_{ j-1 } \ar[r] & e_j \ar[r] & \Sigma q_j
                       }
\]
\end{proof}

\begin{Definition}
\label{def:6}
Let $\cQ$ be a co-slicing in $\sT$.  For $I \subseteq \BR$ we define
a full subcategory of $\sT$ by
\[
  \cQ( I ) = \Big( \bigcup_{ \varphi \in I } \cQ( \varphi ) \Big)^+.
\]
Recall that $( \;\; )^+$ denotes closure under extensions and direct
summands.

As a shorthand, we combine this with inequality signs in an obvious
way; for instance, $\cQ( < a) = \cQ \big(\, ] -\infty , a [ \,\big)$.
\end{Definition}

Definition \ref{def:4}(i) implies
\begin{equation}
\label{equ:SigmaI}
  \Sigma \cQ( I ) = \cQ( \Sigma I )
\end{equation}
where $\Sigma I = \{\, i + 1 \,|\, i \in I \,\}$.  Definition
\ref{def:4}(iii) implies $\cQ( \BR ) = \sT$.

\begin{Lemma}
\label{lem:20}
Let $\cQ$ be a co-slicing in $\sT$.  For $a \leq b$ in $\BR$ we have
\[
  {}^{\perp} \cQ( > b ) \cap \cQ( \leq a )^{\perp}
    = \cQ \big(\, ] a , b ] \,\big).
\]
\end{Lemma}

\begin{proof}
The inclusion $\supseteq$ is clear from Definition \ref{def:4}(ii).

To see $\subseteq$, let
\begin{equation}
\label{equ:lem20a}
  t \in {}^{\perp} \cQ( > b ) \cap \cQ( \leq a )^{\perp}
\end{equation}
and consider the diagrams from Definition \ref{def:4}(iii) and Lemma
\ref{lem:9}.  The lemma implies
\begin{equation}
\label{equ:lem20b}
  t_j \in \cQ \big(\, [ \varphi_1 , \varphi_j ] \,\big)
  \;\; \mbox{and} \;\;
  e_j \in \cQ \big(\, [ \varphi_{ j+1 } , \varphi_n ] \,\big).
\end{equation}

If $b < \varphi_n$ then let $\ell$ be minimal with $b < \varphi_{\ell
  + 1}$.  Then $\sT( t , e_{\ell} ) = 0$ by \eqref{equ:lem20a} and
\eqref{equ:lem20b} so the distinguished triangle $\Sigma^{-1}e_{\ell}
\rightarrow t_{\ell} \rightarrow t$ is split and we have $t_{\ell}
\cong t \oplus t'$ where $t' = \Sigma^{-1}e_{\ell}$.  Truncating the
diagram from Definition \ref{def:4}(iii) gives
\begin{equation}
\label{equ:lem8}
\!\!\!\!\!\!\!\!\!\!\!\!\!\!\!\!\!\!\!\!\!\!
\vcenter{
  \xymatrix @-2.5pc @! {
    0 \cong t_0 \ar[rr] & & t_1 \ar[rr] \ar[dl] & & t_2 \ar[rr] \ar[dl] & & \cdots \ar[rr] & & t_{\ell - 1} \ar[rr] & & t_{\ell} \; \lefteqn{\cong t \oplus t'} \ar[dl] \\
    & q_1 \ar@{~>}[ul] & & q_2 \ar@{~>}[ul] & & & & & & q_{\ell} \ar@{~>}[ul] & \\
                       }
        }
\end{equation}
with $q_j \in \cQ( \varphi_j )$ and $\varphi_1 < \cdots <
\varphi_{\ell} \leq b$.  If $\varphi_n \leq b$ then diagram
\eqref{equ:lem8} also exists with $\ell = n$ and $t' = 0$.

If $a < \varphi_1$ then diagram \eqref{equ:lem8} shows $t \oplus t'
\in \cQ \big(\, ] a , b ] \,\big)$ whence $t \in \cQ \big(\, ] a , b ]
\,\big)$ as desired.

If $\varphi_1 \leq a$ then let $m$ be maximal with $\varphi_m \leq a$.
By Lemma \ref{lem:9} applied to diagram \eqref{equ:lem8} there is a
distinguished triangle
\[
  t_m \rightarrow t \oplus t' \rightarrow f_m
\]
with
\begin{equation}
\label{equ:lem8b}
  f_m \in \big( \cQ( \varphi_{ m+1 } ) \cup \cdots \cup \cQ( \varphi_{\ell} ) \big)^-  
  \subseteq \cQ \big(\, [ \varphi_{ m+1 } , \varphi_{\ell} ] \,\big)
  \subseteq \cQ \big(\, ] a , b ] \,\big).
\end{equation}
We have $\sT( t_m , t ) = 0$ by \eqref{equ:lem20a} and
\eqref{equ:lem20b}, so the distinguished triangle is isomorphic to the
direct sum of distinguished triangles $0 \rightarrow t
\stackrel{=}{\rightarrow} t$ and $t_m \rightarrow t' \rightarrow
f'_m$.  Hence $f_m \cong t \oplus f'_m$ and so $t \in \cQ \big(\, ] a
, b ] \,\big)$ by equation \eqref{equ:lem8b}.
\end{proof}

\begin{Remark}
\label{rmk:20}
By changing the inequalities suitably, the proof also shows
\[
  {}^{\perp} \cQ( > b ) \cap \cQ( < a )^{\perp}
    = \cQ \big(\, [ a , b ] \,\big).
\]
\end{Remark}

The next lemma makes the formal connection to co-t-structures.  It is
analogous to the last part of \cite[sec.\ 3]{B}.

\begin{Lemma}
\label{lem:8}
If $\cQ$ is a co-slicing in $\sT$ then $\big( \cQ( \leq 1 ) , \cQ( > 1
) \big)$ is a bounded co-t-structure in $\sT$ with co-heart $\cQ
\big(\, ] 0 , 1 ] \,\big)$.
\end{Lemma}

\begin{proof}
The co-t-structure: We must check Definition \ref{def:co-t-structure}.
The subcategories $\cQ( \leq 1 )$ and $\cQ( > 1 )$ are full and closed
under direct sums and summands by definition.  Definition
\ref{def:co-t-structure}(i) follows from equation \eqref{equ:SigmaI}.
Definition \ref{def:co-t-structure}(ii) follows from Definition
\ref{def:4}(ii).  And Definition \ref{def:co-t-structure}(iii) follows
from Lemma \ref{lem:9}. 

Boundedness: Clear by Definition \ref{def:4}, parts (i) and (iii).

The co-heart: In a co-t-structure $( \sA , \sB )$ we have $\sA =
{}^{\perp}\sB$ and $\sB = \sA^{\perp}$ whence $\Sigma^{-1}\sB =
(\Sigma^{-1} \sA)^{\perp}$, so the co-heart is $\sC = \sA \cap
\Sigma^{-1}\sB = {}^{\perp}\sB \cap (\Sigma^{-1} \sA)^{\perp}$.
Inserting the co-t-structure of this lemma gives $\sC = \cQ \big(\, ] 
0 , 1 ] \,\big)$ by Lemma \ref{lem:20}.
\end{proof}

\begin{Remark}
\label{rmk:26}
Let $\cQ$ be a co-slicing in $\sT$ and let $a < b \leq a+1$ in
$\BR$.  Then
\[
  \cQ \big(\, ] a , b ] \,\big)
  = \add \Big( \bigcup_{\varphi \in ] a , b ]} \cQ( \varphi ) \Big).
\]
The inclusion $\supseteq$ is clear, and $\subseteq$ holds because the
right hand side is closed under extensions.  In fact, any extension
between two of its objects is trivial because of Definition
\ref{def:4}, parts (i) and (ii).
\end{Remark}

\begin{Remark}
\label{rmk:26b}
Let $\cQ$ be a co-slicing in $\sT$.  Lemma \ref{lem:8} and Remark
\ref{rmk:26} imply that
\[
  \sC = \add \Big( \bigcup_{\varphi \in ] 0 , 1 ]} \cQ( \varphi ) \Big)
\]
is the co-heart of the bounded co-t-structure $\big( \cQ( \leq 1 ) ,
\cQ( > 1 ) \big)$ in $\sT$.  The group $\K_0^{\split}( \sC )$ is free
on a basis consisting of the isomorphism classes of indecomposable
objects in $\displaystyle{\bigcup_{\varphi \in ] 0 , 1 ]}} \cQ(
\varphi )$.  The group is isomorphic to $\K_0( \sT )$ by Proposition
\ref{pro:K} so is finitely generated by assumption.

It follows that $\cQ( \varphi ) \neq 0$ for only finitely many
$\varphi \in \; ] 0 , 1 ]$ and that each $\cQ( \varphi )$ has
only finitely many isomorphism classes of indecomposable objects.

Combining with Definition \ref{def:4}(i) shows that there exists $0 <
\varepsilon_0 < \frac{1}{2}$ such that within each interval $[
\varphi_0 - \varepsilon_0 , \varphi_0 + \varepsilon_0 ]$, there is at
most one $\varphi$ with $\cQ( \varphi ) \neq 0$.
\end{Remark}

\section{The metric space of co-slicings}
\label{sec:metric}

In \cite[sec.\ 6]{B} the set of slicings in a triangulated category
was turned into a metric space, and we do the same for the set
of co-slicings.  The formula in the following definition is due to
\cite[lem.\ 6.1]{B}.

\begin{Definition}
\label{def:21}
If $\cQ$ and $\cR$ are co-slicings in $\sT$, then we set
\[
  d( \cQ , \cR )
  = \inf \big\{\: \varepsilon > 0  \;\; \big| \;\; 
      \cQ( \varphi ) \subseteq 
      \cR \big(\, [ \varphi - \varepsilon , \varphi + \varepsilon ]
          \,\big) 
      \mbox{ for each } \varphi \in \BR 
         \:\big\}.
\]
\end{Definition}

\begin{Remark}
\label{rmk:21}
By Definition \ref{def:4}(i), we can replace $\BR$ by $] 0 , 1 ]$
in the formula without changing the value of $d( \cQ , \cR )$.  
\end{Remark}

\begin{Proposition}
\label{pro:23}
The function $d$ is a metric on the set of co-slicings in $\sT$. 
\end{Proposition}

\begin{proof}
(i) $d( \cQ , \cR ) < \infty$: By Remark \ref{rmk:26b} the subcategory
$\cQ( \varphi )$ is non-zero for only finitely many $\varphi \in
\; ] 0 , 1 ]$, and for each $\varphi$ it has only finitely many
isomorphism classes of indecomposable objects.  Using Definition
\ref{def:4}(iii), this implies that there is an $\varepsilon > 0$
such that $\cQ( \varphi ) \subseteq \cR \big(\, [ \varphi
- \varepsilon , \varphi + \varepsilon ] \,\big)$ for each $\varphi \in
\; ] 0 , 1 ]$.  Hence $d( \cQ , \cR ) \leq \varepsilon$ by Remark
\ref{rmk:21}. 

(ii) $d( \cQ , \cR ) = d( \cR , \cQ )$: Given $\varepsilon > 0$, by
symmetry it is enough to show that if $\cQ( \varphi ) \subseteq
\cR \big(\, [ \varphi - \varepsilon , \varphi + \varepsilon ] \,\big)$
for each $\varphi$ then $\cR( \varphi ) \subseteq \cQ \big(\, [
\varphi - \varepsilon , \varphi + \varepsilon ] \,\big)$ for each
$\varphi$.  By Definition \ref{def:4}(ii), the condition $\cQ( \varphi
) \subseteq
\cR \big(\, [ \varphi - \varepsilon , \varphi + \varepsilon ] \,\big)$
for each $\varphi$ implies $\sT \big(\, \cR( \varphi ) , \cQ( >
\varphi + \varepsilon ) \,\big) = 0$ for each $\varphi$.  That is,
\[
  \cR( \varphi ) \subseteq {}^{\perp}\cQ( > \varphi + \varepsilon )
  \mbox{ for each } \varphi.
\]
Similarly, the condition implies 
\[
  \cR( \varphi ) \subseteq \cQ( < \varphi - \varepsilon )^{\perp}
  \mbox{ for each } \varphi.
\]
Together these inclusions imply $\cR( \varphi ) \subseteq \cQ \big(\,
[ \varphi - \varepsilon , \varphi + \varepsilon ] \,\big)$ for each
$\varphi$ by Remark \ref{rmk:20}.

(iii) $d( \cQ , \cS ) \leq d( \cQ , \cR ) + d( \cR , \cS )$: If $d( \cQ
, \cR ) = x$ and $d( \cR , \cS ) = y$ then there are inclusions $\cQ(
\varphi ) \subseteq \cR \big(\, [ \varphi - x - \delta , \varphi + x +
\delta ] \,\big)$ and $\cR( \varphi ) \subseteq \cS \big(\, [ \varphi
- y - \delta , \varphi + y + \delta ] \,\big)$ for each $\varphi \in
\BR$ and $\delta > 0$.  They clearly imply $\cQ( \varphi ) \subseteq
\cS \big(\, [ \varphi - (x + y) - 2\delta , \varphi
+ (x + y) + 2\delta ] \,\big)$ whence $d( \cQ , \cS ) \leq x +
y$. 

(iv) $d( \cQ , \cR ) = 0 \Rightarrow \cQ = \cR$: Let $q \in \cQ(
\varphi )$ be given.  When $d( \cQ , \cR ) = 0$, then $q \in \cR
\big(\, [ \varphi - \varepsilon , \varphi + \varepsilon ] \,\big)$ for
each $\varepsilon > 0$.  This implies $q \in {}^{\perp} \cR( > \varphi
) \cap \cR( < \varphi )^{\perp}$ by Definition \ref{def:4}(ii) whence
Remark \ref{rmk:20} gives $q \in \cR \big(\, [ \varphi , \varphi ]
\,\big) = \cR( \varphi )$.  So $\cQ( \varphi ) \subseteq \cR( \varphi
)$ and the opposite inclusion holds by symmetry.
\end{proof}

\section{Co-stability functions}
\label{sec:co-stability_functions}

This section introduces co-stability functions and the split
Harder-Narasimhan pro\-per\-ty.  They are analogues of the stability
functions and the Harder-Narasimhan property of \cite[sec.\ 2]{B}, and
will permit us to show that the co-stability manifold is divided into
subsets corresponding to bounded co-t-structures; see Remark
\ref{rmk:chambers}.

\begin{Definition}
\label{def:co-stability_function}
A co-stability function on an additive category $\sA$ is a group
homomorphism
\[
  Z : \K_0^{\split}( \sA ) \rightarrow \BC
\]
such that $Z( a ) \in H$ for each object $a \neq 0$, where
\[
  H = \{\, r \exp( i\pi\varphi ) \,|\, 0 < r, \, 0 < \varphi \leq 1 \,\}
\]
is the strict upper half plane.

The phase $\varphi( a )$ of an object $a \neq 0$ is the unique element
in $]0,1]$ for which $Z( a ) = r \exp \big( i\pi\varphi( a ) \big)$.
\end{Definition}

We need a split version of Harder-Narasimhan theory so we would like to define
an object $a \neq 0$ to be $Z$-semistable if $a', a'' \neq 0$ and $a =
a' \oplus a''$ implies $\varphi( a' ) \leq \varphi( a )$.  However, this
is equivalent to the following definition.

\begin{Definition}
\label{def:semistable}
Let $Z$ be a co-stability function on the additive category $\sA$.
An object $a \neq 0$ is called $Z$-semistable if $a = a' \oplus a''$
with $a' \neq 0$ implies $\varphi( a' ) = \varphi( a )$.
\end{Definition}

If $\sA$ is Krull-Schmidt, then $a \neq 0$ is $Z$-semistable if and
only if its indecomposable direct summands have the same phase.

\begin{Definition}
\label{def:HN}
A co-stability function $Z$ on an additive category $\sA$ is said to
have the split Harder-Narasimhan property if it satisfies the
following.
\begin{enumerate}

  \item  If $a_1, a_2 \neq 0$ are $Z$-semistable with $\varphi( a_1 ) <
\varphi( a_2 )$, then we have $\sA( a_1 , a_2 ) = 0$.

\smallskip

  \item  Each $a \neq 0$ can be written $a = a_1 \oplus \cdots \oplus
a_s$ with the $a_i$ being $Z$-semistable and $\varphi( a_1 ) < \cdots <
\varphi( a_s )$. 

\end{enumerate}
\end{Definition}

If $\sA$ is Krull-Schmidt, then (ii) is vacuous but (i) is usually
not.

\section{Co-stability conditions}
\label{sec:co-stability_conditions}

This section introduces co-stability conditions and proves a
separation result in Proposition \ref{pro:27}.  We also show the
precise relationship between co-stability conditions and
co-t-structures in Proposition \ref{pro:17}.  These results are
analogues of \cite[lem.\ 6.4]{B} and \cite[prop.\ 5.3]{B}.

\begin{Definition}
\label{def:co-stability_condition}
A co-stability condition on $\sT$ is a pair $( Z , \cQ )$, where $Z :
\K_0( \sT ) \rightarrow \BC$ is a group homomorphism and $\cQ$ a
co-slicing in $\sT$, such that
\[
  0 \neq q \in \cQ( \varphi )
  \; \Rightarrow \;
  Z( q ) = m( q ) \exp( i\pi\varphi )
\]
with $m( q ) > 0$.
\end{Definition}

\begin{Proposition}
\label{pro:27}
If $( Z , \cQ )$ and $( Z , \cR )$ are co-stability conditions in
$\sT$ and $d( \cQ , \cR ) < \frac{1}{2}$, then $\cQ = \cR$. 
\end{Proposition}

\begin{proof}
When $d( \cQ , \cR ) < \frac{1}{2}$ holds, Definition \ref{def:21}
implies that there is $\varepsilon < \frac{1}{2}$ such that $\cQ(
\varphi ) \subseteq \cR \big(\, ] \varphi - \varepsilon , \varphi +
\varepsilon ] \,\big)$ for each $\varphi$.  That is,
\[
  \cQ( \varphi ) \subseteq
  \add \Big( \bigcup_{\psi \in ] \varphi - \varepsilon , \varphi + \varepsilon ]} \cR( \psi ) \Big)
\]
for each $\varphi$ by Remark \ref{rmk:26}.  So if $q \in \ind \cQ(
\varphi )$ then $q \in \ind \cR( \psi )$ for a $\psi \in \; ] \varphi
- \varepsilon , \varphi + \varepsilon ]$.  Since $( Z , \cQ )$ and $(
Z , \cR )$ are co-stability conditions we get $Z( q ) = m( q ) \exp(
i\pi\varphi )$ and $Z( q ) = m'(q) \exp( i\pi\psi )$ with $m( q ) , m(
q' ) > 0$, and then $\psi = \varphi$ since $\varepsilon <
\frac{1}{2}$.  Hence $q \in \ind \cR( \varphi )$ and we learn $\cQ(
\varphi ) \subseteq \cR( \varphi )$.  The opposite inclusion holds by
symmetry.
\end{proof}

\begin{Proposition}
\label{pro:17}
Giving a co-stability condition on $\sT$ is equivalent to giving a
bounded co-t-structure in $\sT$ and a co-stability function on its
co-heart which has the split Harder-Narasimhan property.
\end{Proposition}

\begin{proof}
We describe how to map back and forth.

(i)  Let $( Z , \cQ )$ be a co-stability condition on $\sT$.  Then
$\big( \cQ( \leq 1 ) , \cQ( > 1 ) \big)$ is a bounded co-t-structure
in $\sT$ by Lemma \ref{lem:8}.  If $\sC$ is the co-heart then
Proposition \ref{pro:K} gives an isomorphism $\K_0^{\split}( \sC )
\rightarrow \K_0( \sT )$ so $Z$ can be viewed as a group homomorphism 
$Z : \K_0^{\split}( \sC ) \rightarrow \BC$.  This is a co-stability
function on $\sC$ which has the split Harder-Narasimhan property.

(ii) Conversely, let $( \sA , \sB )$ be a bounded co-t-structure in
$\sT$ with co-heart $\sC$, and let $Z$ be a co-stability function on
$\sC$ which has the split Harder-Narasimhan property.  Proposition
\ref{pro:K} means that $Z$ can be viewed as a group homomorphism $Z :
\K_0( \sT ) \rightarrow \BC$.  For $0 < \varphi \leq 1$, let $\cQ(
\varphi )$ be the full subcategory consisting of $0$ and all objects
in $\sC$ which are $Z$-semistable of phase $\varphi$; extend to other
values of $\varphi$ by setting $\cQ( \varphi + 1 ) = \Sigma \cQ(
\varphi )$.  Then $( Z , \cQ )$ is a co-stability condition.
\end{proof}

\section{Two triangle lemmas}
\label{sec:triangles}

The following two lemmas are easy consequences of the octahedral axiom
and we omit the proofs.

\begin{Lemma}
\label{lem:36}
Consider the following diagram in $\sT$ consisting of two
distinguished triangles. 
\[
  \xymatrix @-0.5pc @! {
    t_0 \ar[rr] & & t_1 \ar[rr] \ar[dl] & & t_2 \ar[dl] \\
    & c_1 \ar@{~>}[ul] & & c_2 \ar@{~>}[ul] \\
               }
\]
If $\sT( c_2 , \Sigma c_1 ) = 0$, then there is a distinguished
triangle 
\[
  \xymatrix @-2.5pc @! {
    t_0 \ar[rr] & & t_2 \ar[rr] & & c_1 \oplus c_2. \\
               }
\]
\end{Lemma}

\begin{Lemma}
\label{lem:37}
Consider the following distinguished triangle in $\sT$.
\[
  \xymatrix @-2.25pc @! {
    t_0 \ar[rr] & & t_2 \ar[rr] & & c_1 \oplus c_2 \\
               }
\]
There is a diagram consisting of two distinguished triangles, 
\[
\vcenter{
  \xymatrix @-0.5pc @! {
    t_0 \ar[rr] & & t'_1 \ar[rr] \ar[dl] & & t_2 \ar[dl] \\
    & c_2 \ar@{~>}[ul] & & c_1 \ar@{~>}[ul] \\
               }
        }.
\]
\end{Lemma}

\section{The co-stability manifold}
\label{sec:manifold}

This section proves a deformation result in Proposition
\ref{pro:deformation}; it is an analogue of \cite[thm.\ 7.1]{B}.  As
in \cite{B}, by combining with a separation result, in our case
Proposition \ref{pro:27}, one obtains a manifold as a formal
consequence.  We formulate this as Theorem \ref{thm:manifold} which
contains Theorem A.

An important ingredient is the following technical condition on
separation which plays a role analogous to local-finiteness in
\cite{B}.

\begin{Definition}
\label{def:40}
A co-slicing $\cQ$ of $\sT$ is said to satisfy condition (S) if
\[
  q_1, q_2 \in \ind \cQ( \varphi ), \: q_1 \not\cong q_2
  \;\Rightarrow\; \sT( q_1 , q_2 ) = 0
\]
for each $\varphi$.
\end{Definition}

Let us write $\K_0( \sT )^* = \Hom_{\BZ} \big( \K_0( \sT ) , \BC
\big)$.  Since $\K_0( \sT )$ is finitely generated, $\K_0( \sT )^*$ is a
finite dimensional vector space over $\BC$; it can be equipped with
the usual topology.  Let $\Coslice( \sT )$ denote the set of
co-slicings of $\sT$ satisfying condition (S); it is a metric space by
Proposition \ref{pro:23} so in particular a topological space.
Consider the product space $\K_0( \sT )^* \times \Coslice( \sT )$.

\begin{Definition}
\label{def:Costab}
The co-stability manifold of $\sT$ is the topological subspace
\[
  \Costab( \sT ) \subseteq \K_0( \sT )^* \times \Coslice( \sT )
\]
consisting of co-stability conditions $( Z , \cQ )$.
\end{Definition}

The definition is motivated by the following theorem.

\begin{Theorem}
\label{thm:manifold}
The topological space $\Costab( \sT )$ is a topological manifold of
dimension $2n$ where $n = \rank \K_0( \sT )$. 
\end{Theorem}

As mentioned, this is a formal consequence of results on separation
and deformation which imply that the canonical map $\Costab( \sT )
\rightarrow \K_0( \sT )^*$ given by $( Z , \cQ ) \mapsto Z$ is a local
homeomorphism.  In our case, separation is by Proposition \ref{pro:27}
while deformation takes the following form.

\begin{Proposition}
\label{pro:deformation}
Let $( Z , \cQ ) \in \Costab( \sT )$ be given and let $0 <
\varepsilon_0 < \frac{1}{2}$ be as in Remark \ref{rmk:26b}. 

Assume that $0 < \varepsilon \leq \varepsilon_0$ and that $W \in
\K_0( \sT )^*$ satisfies
\[
  | W( q ) - Z( q ) | < \sin( \pi\varepsilon ) \, | Z( q ) |
\]
for each $q \in \cQ( \varphi ) \setminus 0$ with $\varphi \in \BR$.

Then there is $( W , \cR ) \in \Costab( \sT )$ such that $d( \cQ , \cR
) < \varepsilon$.
\end{Proposition}

\begin{proof}
For $\psi \in \BR$ we define $\cR( \psi )$ as the full subcategory of
$\sT$ which is closed under direct sums and summands and has the
following indecomposable objects.
\[
  \ind \cR( \psi ) =
  \Bigg\{\, q \in \ind \cQ( \varphi ) \,\Bigg|\,
         \begin{array}{l}
           \psi - \varepsilon < \varphi < \psi + \varepsilon, \\[1.5mm]
           W( q ) = m'( q ) \exp( i\pi\psi ) 
           \mbox{ with } m'( q ) > 0 
         \end{array} \,\Bigg\}
\]
We will show that $\cR$ is a co-slicing satisfying condition
(S).

Definition \ref{def:4}(i) is clear for $\cR$.

Definition \ref{def:4}(ii) and condition (S):  Let $r_j \in \ind \cR(
\psi_j )$ for $j = 1 , 2$ and assume either $\psi_1 < \psi_2$ (for
Definition \ref{def:4}(ii)) or $\psi_1 = \psi_2$ and $r_1 \not\cong
r_2$ (for condition (S)).  
By definition, we have $r_j \in \ind \cQ( \varphi_j
)$ with 
\begin{equation}
\label{equ:inequality}
\psi_j - \varepsilon < \varphi_j < \psi_j + \varepsilon \mbox{ for } j=1,2.
\end{equation}
We split into three cases.

$\varphi_1 < \varphi_2$:  Then $\sT( r_1 , r_2 ) = 0$ by Definition
\ref{def:4}(ii) for $\cQ$.

$\varphi_1 = \varphi_2$: There are two possibilities.  First, we may
have $\psi_1 = \psi_2$.  Then $r_1 \not\cong r_2$ by assumption whence
$\sT( r_1 , r_2 ) = 0$ by condition (S) for $\cQ$.  Secondly, we may
have $\psi_1 < \psi_2$.  We also have $\psi_2 < \psi_1 + 2\varepsilon$ by inequality \eqref{equ:inequality},
and $2\varepsilon < 2\varepsilon_0 < 1$, so $W( r_j ) = m'( r_j )
\exp( i\pi\psi_j )$ implies $W( r_1 ) \neq W( r_2 )$.  But then $r_1
\not\cong r_2$ whence $\sT( r_1 , r_2 ) = 0$ by condition (S) for
$\cQ$.

$\varphi_1 > \varphi_2$: The inequality \eqref{equ:inequality} also gives $\varphi_1 < \varphi_2 +
2\varepsilon$, so $\varphi_1$ is certainly in the interval $[
\varphi_2 , \varphi_2 + 2\varepsilon_0 ]$ and so is $\varphi_2$.  But
by Remark \ref{rmk:26b} each closed interval of length
$2\varepsilon_0$ contains at most one $\varphi$ with $\cQ( \varphi )
\neq 0$.  This gives a contradiction with $\varphi_1 \neq \varphi_2$
and $r_j \in \ind \cQ( \varphi_j )$.

Definition \ref{def:4}(iii): We start with an observation.  If $q \in
\ind \cQ( \varphi )$, then the inequality in the proposition implies
$W( q ) = m'( q ) \exp( i\pi\psi )$ with $m'( q ) > 0$ and $\psi$
satisfying $\psi - \varepsilon < \varphi < \psi + \varepsilon$, whence
$q \in \cR( \psi )$.

Now let $t \neq 0$ in $\sT$.  Using that $\cQ$ is a co-slicing, pick a
diagram as in Definition \ref{def:4}(iii).  Using Lemma \ref{lem:37},
each distinguished triangle in the diagram can be refined to a
sequence of distinguished triangles with indecomposable third term.
This gives a diagram
\begin{equation}
\label{equ:deformation}
\vcenter{
  \xymatrix @-2.7pc @! {
    0 \cong t_0 \ar[rr] & & t_1 \ar[rr] \ar[dl] & & t_2 \ar[rr] \ar[dl] & & \cdots \ar[rr] & & t_{ p-1 } \ar[rr] & & t_p \cong t \ar[dl] \\
    & q_1 \ar@{~>}[ul] & & q_2 \ar@{~>}[ul] & & & & & & q_p \ar@{~>}[ul] & \\
                       }
        }
\end{equation}
consisting of distinguished triangles where $q_j \in \ind \cQ(
\varphi_j )$.

By the above observation, we have $q_j \in \cR( \psi_j )$ for each $j$
for certain $\psi_j \in \BR$.  Suppose that $\psi_j > \psi_{ j+1 }$
for some $j$.  Then $\sT( q_{ j+1 } , \Sigma q_j ) = 0$ by Definition
\ref{def:4}, parts (i) and (ii), which we have already shown for
$\cR$.  So Lemmas \ref{lem:36} and \ref{lem:37} imply that in diagram
\eqref{equ:deformation}, the part
\[
  \xymatrix @-1.5pc @! {
    t_{ j-1 } \ar[rr] & & t_j \ar[rr] \ar[dl] & & t_{ j+1 } \ar[dl] \\
    & q_j \ar@{~>}[ul] & & q_{ j+1 } \ar@{~>}[ul] \\
               }
\]
can be replaced with
\[
\vcenter{
  \xymatrix @-1.5pc @! {
    t_{ j-1 } \ar[rr] & & t'_j \ar[rr] \ar[dl] & & t_{ j+1 } \ar[dl] \\
    & q_{ j+1 } \ar@{~>}[ul] & & q_j \ar@{~>}[ul] \\
               }
        }.
\]
Repeating this procedure reorders the $q_j$ according to
non-decreasing values of $\psi_j$.  That is, it turns diagram
\eqref{equ:deformation} into a diagram
\[
  \xymatrix @-2.5pc @! {
    0 \cong t_0 \ar[rr] & & t'_1 \ar[rr] \ar[dl] & & t'_2 \ar[rr] \ar[dl] & & \cdots \ar[rr] & & t'_{ p-1 } \ar[rr] & & t_p \cong t \ar[dl] \\
    & r_1 \ar@{~>}[ul] & & r_2 \ar@{~>}[ul] & & & & & & r_p \ar@{~>}[ul] & \\
                       }
\]
consisting of distinguished triangles, where $r_j \in \ind \cR( \psi_j
)$ and $\psi_1 \leq \cdots \leq \psi_p$.  Neighbouring objects $r_j$
and $r_{ j+1 }$ with $\psi_j = \psi_{ j+1 }$ have $\sT( r_{ j+1 } ,
\Sigma r_j ) = 0$, again by Definition \ref{def:4}, parts (i) and
(ii), so $r_j$ and $r_{ j+1 }$ can be collected using Lemma
\ref{lem:36}.  This finally results in the desired diagram
establishing Definition \ref{def:4}(iii) for $\cR$.

To complete the proof, we must show that $( W , \cR )$ is a
co-stability condition and that $d( \cQ , \cR ) < \varepsilon$.  The
former is clear by the definition of $\cR$.  For the latter, note that
by Remark \ref{rmk:26b}, if $\psi$ is given then there are only
finitely many $\varphi$ with $\psi - \varepsilon < \varphi < \psi +
\varepsilon$ and $\cQ( \varphi ) \neq 0$.  Hence there is an
$\varepsilon' < \varepsilon$ such that it makes no difference to
replace $\varepsilon$ by $\varepsilon'$ in the definition of $\ind
\cR( \psi )$, and so $\cR( \psi ) \subseteq \cQ
\big(\, [ \psi - \varepsilon' , \psi + \varepsilon' ] \,\big)$.  This
applies to each of the finitely many $\psi \in \; ] 0 , 1 ]$ for which
$\cR( \psi ) \neq 0$; see Remark \ref{rmk:26b} again.  But then $d(
\cQ , \cR ) < \varepsilon$ by Remark \ref{rmk:21}.
\end{proof}

\begin{Remark}
\label{rmk:chambers}
By Proposition \ref{pro:17}, each point $( Z , \cQ ) \in \Costab( \sT
)$ corresponds to a pair consisting of a bounded co-t-structure in
$\sT$ and a co-stability function on its co-heart which has the split
Harder-Narasimhan property.  In particular, $\Costab( \sT )$ is
divided into subsets corresponding to different co-t-struc\-tu\-res in
$\sT$.
\end{Remark}

\section{Two group actions on the co-stability manifold}
\label{sec:group_actions}

Like the stability manifold, the co-stability manifold admits
commuting continuous left and right actions of the groups $\Aut( \sT
)$ and $\widetilde{\GL}^+( 2 , \BR )$, where $\Aut( \sT )$ is the
group of equivalence classes of triangulated autoequivalences of $\sT$
and $\widetilde{\GL}^+( 2 , \BR )$ is the universal cover of $\GL^+( 2
, \BR )$, the group of $2 \times 2$ real matrices with positive
determinant.  Indeed, we can just copy the formulae from \cite[lem.\
8.2]{B} as follows.

For $F \in \Aut( \sT )$ and $( Z , \cQ ) \in \Costab( \sT )$, set
\[
  F \cdot ( Z , \cQ ) = ( Z \circ [F]^{-1} , \cQ' )
\]
where $[F] \in \Aut \K_0( \sT )$ is induced by $F$ and $\cQ'( \varphi
) = F \big( \cQ( \varphi ) \big)$.

For $\widetilde{\GL}^+( 2 , \BR )$, we use the same description as in
\cite[sec.\ 8]{B}, so an element is a pair $( T , f )$ where $T :
\BR^2 \rightarrow \BR^2$ is an orientation preserving linear map and
$f : \BR \rightarrow \BR$ is an increasing map satisfying $f( x + 1 )
= f( x ) + 1$, such that the induced maps on $( \BR^2 \setminus 0 ) /
\BR_{ >0 }$ and $\BR / 2\BZ$ are the same when these spaces are
identified with $S^1$.  For $( T , f ) \in \widetilde{\GL}^+( 2 , \BR
)$ and $( Z , \cQ ) \in \Costab( \sT )$, set
\[
  ( Z , \cQ ) \cdot ( T , f ) = ( T^{-1} \circ Z , \cQ'' )
\]
where $\cQ''( \varphi ) = \cQ \big( f( \varphi ) \big)$.

\section{Example: The compact derived category of $k[X] / (X^2)$}
\label{sec:example1}

Let $k$ be an algebraically closed field.  The compact derived
category $\Dc \big( k[ X ] / ( X^2 ) \big)$ of the dual numbers over
$k$ is the special case $w = 0$ of $\sU$ in the next theorem, so
Theorem B in the introduction follows.

\begin{Theorem}
\label{thm:example1}
Let $w \leq 0$ be an integer and let $\sU$ be a $k$-linear algebraic
triangulated category which is idempotent complete and classically
generated by a $w$-spherical object; see \cite{HJY}.

The stability manifold of $\sU$ is the empty set.  The co-stability
manifold of $\sU$ is $\BC$.
\end{Theorem}

\begin{proof}
By \cite[thm.\ A]{HJY}, the category $\sU$ has no non-trivial
t-structures, hence no bounded t-structures.  It follows by
\cite[prop.\ 5.3]{B} that it has no stability conditions, so the
stability manifold is the empty set.

By \cite[thm.\ A]{HJY} again, the category $\sU$ has a canonical
co-t-structure $( \sA , \sB )$, and the non-trivial co-t-structures in
$\sU$ are precisely the (de)suspensions $( \Sigma^j \sA , \Sigma^j \sB
)$ for $j \in \BZ$.  The explicit description of the canonical
co-t-structure in \cite[sec.\ 4.e]{HJY} shows that each of the
(de)suspensions is bounded.  It also shows that the co-heart $\sC =
\sA \cap \Sigma^{-1}\sB$ is equal to $\add( c )$ for a certain
indecomposable object $c$.  Hence the co-heart of $( \Sigma^j \sA ,
\Sigma^j \sB )$ is $\Sigma^j \sC = \add( \Sigma^j c )$.

Combining this with Proposition \ref{pro:17} shows that giving a
co-stability condition on $\sU$ is equivalent to giving two pieces of
data: (i) An integer $j$ specifying a bounded co-t-structure $(
\Sigma^j \sA , \Sigma^j \sB )$, and (ii) an element $z$ of the strict
upper half plane $H$ specifying a co-stability function on the
co-heart as follows.
\[
  Z : \K_0^{\split}( \Sigma^j \sC ) \rightarrow \BC, 
  \;\; Z( \Sigma^j c ) = z.
\]
The split Harder-Narasimhan property holds for $Z$ because $\Sigma^j
\sC$ has only one isomorphism class of indecomposable objects.

By the proof of Proposition \ref{pro:17}, these data correspond to the
following co-stability condition $( Z , \cQ )$: By means of
Proposition \ref{pro:K}, the above $Z$ is viewed as a group
homomorphism $Z : \K_0( \sU ) \rightarrow \BC$; it still satisfies $Z(
\Sigma^j c ) = z$.  And writing $z = r \exp( i\pi\varphi )$ with $r >
0$, $\varphi \in \; ] 0 , 1 ]$, the co-slicing $\cQ$ is given by $\cQ(
\varphi ) = \add( \Sigma^j c )$ and $\cQ$ equal to zero on the rest of
the interval $] 0 , 1 ]$.

This co-stability condition can also be described by giving the unique
$\varphi_0 \in \BR$ for which $\cQ( \varphi_0 ) = \add( c )$, along
with the complex number $Z( c ) = z_0$ which has the form $z_0 = s
\exp( i\pi\varphi_0 )$ for some $s > 0$.  Abusing notation, we write
$( Z , \cQ ) = ( z_0 , \varphi_0 )$.

Each co-stability condition clearly satisfies condition (S).

Let $G$ be the closed subgroup of $\widetilde{\GL}^+( 2 , \BR )$
consisting of elements $( T , f )$ where $T$ is the composition of a
rotation and a scaling by a positive real number.  Note that $f( x ) =
x + a$ where $a$ is a real number determined modulo $2\BZ$ by $T$.
Since $G$ is a subgroup of $\widetilde{\GL}^+( 2 , \BR )$, it acts
continuously on $\Costab( \sU )$ by Section \ref{sec:group_actions}.
The action is given by
\[
  ( z_0 , \varphi_0 ) \cdot ( T , f ) = ( T^{-1}z_0 , \varphi_0 - a ).
\]
It is easy to see that the action is free and transitive, so $\Costab(
\sU )$ is homeomorphic to $G$.

However, $G$ is simply connected and $2\BZ$ can be viewed as the
discrete subgroup consisting of the elements $( \id , x \mapsto x + y
)$ for $y \in 2\BZ$.  Hence $G$ is the universal covering group of $G
/ 2\BZ$,
but $G / 2\BZ$ can be identified with the subgroup of $\GL^+( 2 , \BR
)$ consisting of transformations $T$ which are the composition of a
rotation and a scaling by a positive real number.  Hence $G / 2\BZ$ is
homeomorphic to $\BC \setminus 0$, so $G$ is homeomorphic to the
universal cover which is $\BC$.
\end{proof}

\section{Example: The compact derived category of $kA_2$.  Why
Condition (S) is necessary}
\label{sec:example2}

This section shows that without condition (S) of Definition
\ref{def:40}, the conclusion of our deformation result Proposition
\ref{pro:deformation} fails.

Let $k$ be an algebraically closed field.  The Auslander-Reiten quiver
of the compact derived category $\sV = \Dc( kA_2 )$ is $\BZ A_2$.  Let
$x$ and $y$ be consecutive indecomposable objects on the quiver; then
$\K_0( \sV )$ is free on the generators $[x]$ and $[y]$.
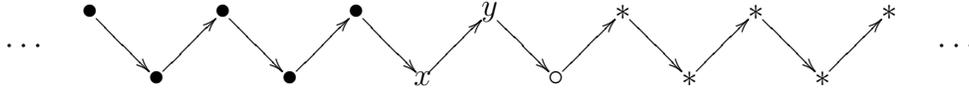
\[
  \xymatrix @-2.7pc @! {
    && *{\bullet} \ar[ddrr] &&&& *{\bullet} \ar[ddrr] &&&& *{\bullet} \ar[ddrr] &&&& *{y} \ar[ddrr] &&&& *{\ast} \ar[ddrr] &&&& *{\ast} \ar[ddrr] &&&& *{\ast} && \\
  \cdots &&&&&&&&&&&&&&&&&&&&&&&&&&&& \cdots \\
    &&&& *{\bullet} \ar[uurr] &&&& *{\bullet} \ar[uurr] &&&& *{x} \ar[uurr] &&&& *{\circ} \ar[uurr] &&&& *{\ast} \ar[uurr] &&&& *{\ast} \ar[uurr] &&&& \\
                       }
\]

Let $\sA$ denote $\add$ of the indecomposable objects forming the left
hand part of the quiver ending at $y$; some of them are marked with
bullets in the sketch.  Let $\sB$ denote $\add$ of the indecomposable
objects forming the right hand part of the quiver, starting with the
asterisks in the sketch.  It is not hard to check that $( \sA , \sB )$
is a bounded co-t-structure in $\sV$.  The co-heart is $\sC = \sA \cap
\Sigma^{-1}\sB = \add ( x , y )$.

Define a co-stability function $Z : \K_0^{\split}( \sC ) \rightarrow
\BC$ by $Z( x ) = Z( y ) = \exp( i\pi\frac{1}{2} )$; it clearly has
the split Harder-Narasimhan property.

By the proof of Proposition \ref{pro:17}, the data $( \sA , \sB )$ and
$Z$ correspond to the following co-stability condition $( Z , \cQ )$:
By means of Proposition \ref{pro:K}, the above $Z$ is viewed as a
group homomorphism $Z : \K_0( \sV ) \rightarrow \BC$; it still
satisfies $Z( x ) = Z( y ) = \exp( i\pi\frac{1}{2} )$.  The co-slicing
$\cQ$ is given by $\cQ( \frac{1}{2} ) = \add( x , y )$ and $\cQ(
\varphi ) = 0$ for $\varphi \in \; ] 0 , 1 ] \setminus \frac{1}{2}$.

Let $\varepsilon < \frac{1}{2}$ be given and let $W \in \K_0( \sV )^*$
be the deformation of $Z$ defined by $W( x ) = \exp( i\pi\frac{1}{2}
)$ and $W( y ) = \cos( \pi\varepsilon) \exp \big( i\pi( \frac{1}{2}
+ \varepsilon ) \big)$.  This $W$ is chosen to satisfy two criteria:
(i) Compared to $Z$, it fixes $x$ but moves $y$ from phase
$\frac{1}{2}$ to phase $\frac{1}{2} + \varepsilon$; (ii) It satisfies
the inequality in Proposition \ref{pro:deformation} because of the
factor $\cos( \pi\varepsilon )$.

\begin{Proposition}
\begin{enumerate}

  \item  Condition (S) fails for $( Z , \cQ )$.

\smallskip

  \item  The conclusion of Proposition \ref{pro:deformation} fails for
  the deformation $W$.  That is, there is no $( W , \cR ) \in \Costab(
  \sV )$ such that $d( \cQ , \cR ) < \varepsilon$.

\end{enumerate}
\end{Proposition}

\begin{proof}
(i)  This is clear because $x, y \in \cQ( \frac{1}{2} )$ while $\sV( x
, y ) \neq 0$.

(ii)  We show more than formulated, namely, there is no $( W , \cR )
\in \Costab( \sV )$ such that $d( \cQ , \cR ) < \frac{1}{2}$.  For
suppose that there is.  Then we have $d( \cQ , \cR ) < \delta <
\frac{1}{2}$ for some $\delta$ and this gives the first of the
inclusions in the following formula.
\[
  \add( x , y )
  = \cQ( { \textstyle \frac{1}{2} } )
  \subseteq \cR \big(\, [ { \textstyle \frac{1}{2} } - \delta , { \textstyle \frac{1}{2} } + \delta ] \,\big)
  \subseteq \cR \big(\, ] 0 , 1 ] \,\big)
  = \add \Big( \bigcup_{\psi \in ] 0 , 1 ]} \cR( \psi ) \Big).
\]
The last equality is by Remark \ref{rmk:26}.

By Remark \ref{rmk:26b}, the right hand side of this formula is the
co-heart of a bounded co-t-structure in $\sV$, so it follows from
Proposition \ref{pro:K} that the right hand side has precisely two
isomorphism classes of indecomposable objects which must necessarily
be the isomorphism classes of $x$ and $y$.

However, since $( W , \cR )$ is a co-stability condition, we have $W(
r ) = m'( r ) \exp( i\pi\psi )$ for $r \in \cR( \psi ) \setminus 0$.
The values $W( x )$ and $W( y )$ hence force $x \in \cR( \frac{1}{2}
)$ and $y \in \cR( \frac{1}{2} + \varepsilon )$.  But this contradicts
$\sV( x , y ) \neq 0$ by Definition \ref{def:4}(ii).
\end{proof}

\medskip
\noindent
{\bf Acknowledgement.}
Part of this work was carried out while J\o rgensen was visiting
Hannover supported by the research priority programme SPP 1388 {\em
Darstellungstheorie} of the Deutsche Forschungsgemeinschaft (DFG).
He gratefully acknowledges financial support through the grant HO
1880/4-1 held by Thorsten Holm.

\end{document}